\documentclass[a4j,10pt]{article}

\usepackage{amsmath}
\usepackage{amssymb}
\usepackage{amsthm}
\usepackage{graphics}
\usepackage[dvips]{graphicx}
\usepackage{amscd}

\numberwithin{equation}{section}
\theoremstyle{definition}
\newtheorem{example}{example}[section]
\newtheorem{df}[example]{Definition}
\theoremstyle{remark}
\newtheorem{rem}[example]{	{\bf Remark}}
\theoremstyle{plain}

\newtheorem{prop}[example]{Proposition}

\newtheorem{thm}[example]{{\bf Theorem}}

{\par\ifdim\lastskip<\bigskipamount\removelastskip\fi\smallskip%
\noindent\begin{trivlist}%
\item[\hskip\labelsep{\bf Proof\ }]}%
{\rule[-.1mm]{1.5mm}{2ex}\end{trivlist}\medskip\par}

\newtheorem{theorem}[example]{Theorem}
\theoremstyle{remark}

\def\Real{{\Bbb R}}

\def\Rc{\varrho} 
\def\rc{\textrm{Rc}} 
\theoremstyle{remark}

\begin{document}
\title{Examples of algebraic Ricci solitons in the Pseudo-Riemannian case} 
\author{Kensuke Onda\thanks{%
The author was supported by funds of Nagoya University and OCAMI. \newline 
} }
\maketitle

 


\begin{abstract} 
This paper provides a study of algebraic Ricci solitons in the pseudo-Riemannian case. 
In the Riemannian case, all nontrivial homogeneous algebraic Ricci solitons are expanding algebraic Ricci solitons. 
In this paper, we obtain a steady algebraic Ricci soliton and a shrinking algebraic Ricci soliton in the Lorentzian setting. 
\end{abstract} 


\section{Introduction} 
The problem to construct a distinguished metric on a manifold is important in differential geometry. 
The Einstein structure and the Ricci soliton structure are candidates. 
For answering this question, the Ricci flow  introduced by Hamilton in \cite{H82} is an important tool. 
Let $g (t)$ be a $1$-parameter family of Riemannian metrics on a differentiable manifold $M^n$. 
If the Riemannian metric $g(t)$ satisfies the equation
$$\frac{\partial }{\partial t} g(t)_{ij} = -2\Rc [g(t)]_{ij}, $$
then $g(t)$ is called a solution to the Ricci flow. 

Hamilton proved that the Ricci flow on a closed manifold has a solution for a short time. 
It is the most basic result on the existence of the Ricci flow solution. 
Hamilton \cite{H82} proved that a closed $3$-manifold with positive Ricci curvature is diffeomorphic to $S^3$. 
Hamilton's idea in \cite{H82} was to use the normalized Ricci flow 
\begin{equation*}
  \frac{\partial }{\partial t} g(t) = \dfrac{2}{n} \Big( \dfrac{\int _{M^n} R d v}{\int _{M^n} d v} \Big) g - 2\Rc 
\end{equation*}
starting with the given Ricci positively curved metric on the given 3-manifold. He proved that 
the solution converges to a constant curvature metric exponentially fast. After Hamilton's work 
\cite{H82}, the study of the Ricci flow has been one of the central problems in differential geometry. 
For instance, Perelman \cite{P02} was able to prove Thurston's geometrization conjecture. 

From the definition of the Ricci flow, a fixed point of the Ricci flow is a Ricci-flat metric, 
and a fixed point of the normalized Ricci flow is an Einstein metric. 
From Proposition \ref{selfsimilar}, Ricci solitons change only by 
diffeomorphism and rescaling, and are regarded as generalized fixed points. 
In other words, 
let $\mathfrak{M} (M^n)$ be the space of Riemannian metrics on $M^n$, 
and $\mathfrak{D} (M^n)$ the diffeomorphism group of $M^n$. 
Considering the dynamical system of the Ricci flow on the moduli space 
$\mathfrak{M} (M^n) / \mathfrak{D} (M^n) \times \Real _+$, we regard Ricci solitons as fixed points. 

In this paper, we study left-invariant pseudo-Riemannian metrics and algebraic Ricci solitons on Loretzian Lie groups.
The concept of an algebraic Ricci soliton was first introduced by Lauret in the Riemannian case (see \cite{L01}). 
Lauret proved that algebraic Ricci solitons on homogeneous Riemannian manifolds are Ricci solitons.  
In general, problems for Ricci solitons are second-order differential equations. 
However, problems for algebraic Ricci solitons are algebraic equations. 
Therefore, algebraic Ricci solitons allow us to construct Ricci solitons in an algebraic way, i.e., using algebraic Ricci soliton theory, 
the study of Ricci solitons on homogeneous manifolds becomes algebraic. 

So far, Ricci solitons have been studied in the Riemannian case. 
Recently, the study of Ricci solitons in the pseudo-Riemannian setting has started with special attention to the Lorentzian case. 
In the Riemannian case, all homogeneous non-trivial Ricci solitons are expanding Ricci solitons. 
In the pseudo-Riemannian case, there are shrinking homogeneous non-trivial Ricci solitons discovered in \cite{O10}, 
while all vector fields of these Ricci solitons are not left-invariant. 
And 3-dimensional homogeneous Lorentzian Ricci solitons with left-invariant vector fields are classified in \cite{BCGG09}. 
Other results about Lorentzian Ricci solitons are found in \cite{BBGG10}, \cite{BGG11}, \cite{CD11}, \cite{CF11}.  

In this paper, we study algebraic Ricci solitons on Lorentzian Lie groups in the Lorentzian case. 
By using algebraic Ricci soliton theory, we can construct homogeneous Lorentzian Ricci solitons in an algebraic way. 
For example, we can construct the Ricci soliton in \cite{O10} by using algebraic Ricci solitons and Theorem \ref{thm}. 
Recall that all homogeneous non-trivial solvsolitons are expanding in the Riemannian setting.
In this paper, we construct Lorentzian algebraic Ricci solitons on the Heisenberg group $H_N$ and the oscillator groups $G_m (\lambda )$ and on three-dimensional Lorentzian Lie groups.  
In particular, we obtain new Lorentzian Ricci solitons on $H_N$ and $G_m (\lambda )$. 

This paper is organized as follows.
In Section 2, we introduce algebraic Ricci solitons in the pseudo-Riemannian case, 
and prove Theorem 2.5 which claims that algebraic Ricci solitons give rise to Ricci solitons.   
Sections 3, 4, and 5 contain Examples of solvsolitons on Lorentzian Lie groups. 
Particularly, in Section 3 and 4, we obtain the new Ricci solitons. 

\section{Algebraic Ricci solitons and Ricci solitons}
The concept of an algebraic Ricci soliton was first introduced by Lauret in the Riemannian case (see \cite{L01}). 
The definition extends to the pseudo-Riemannian case: 
\begin{df}
\textit{Let } $\left( \mathit{G,g}\right) $\textit{\ be a simply connected
Lie group equipped with the left-invariant pseudo-Riemannian metric $g$, and
let } ${\mathfrak{g}}$ denote \textit{the Lie algebra }of $\mathit{G.}$ 
\textit{Then $g$ is called an algebraic Ricci soliton if it satisfies } 
\begin{equation}
\rc =c\mathrm{Id}+D  \label{Soliton}
\end{equation}%
\textit{where }$\rc$\textit{\ denotes the Ricci operator, $c$ is a
real number, and }$\mathrm{D\in Der}\left( \mathfrak{g}\right) $. 
\textit{\ In particular, an algebraic Ricci soliton on a solvable Lie group,
(a nilpotent Lie group) is called a solvsoliton (a nilsoliton).}
\end{df} 
\hspace{-0.8cm} Recall that $\textrm{Der} ({\mathfrak g})$ denotes the Lie algebra of derivations of ${\mathfrak g}$, 
$$\textrm{Der} ({\mathfrak g}) 
:= \{ D\in {\mathfrak g l} ({\mathfrak g}) | D[X, Y] = [DX, Y] + [X, DY] \, \textrm{for any} \, X, Y\in \mathfrak g \} .$$
Obviously, Einstein metrics are algebraic Ricci solitons. 
In this paper, we consider solvsolitons on a solvable Lie group. 
\begin{rem}
Let $G$ be a semisimple Lie group, $g$ a Left-invariant Riemannian metric. 
If $g$ is a solvsoliton, then $g$ is Einstein (see \cite{L01}).  
\end{rem}

Next we introduce Ricci solitons. 
Let $g_0$ be a pseudo-Riemannian metric on a manifold $M^n$.
If $g_0$ satisfies
$$ \Rc [g_0] = c g_0 + L_X g_0  \ ,$$ 
where $\Rc$ is the Ricci tensor of $g_0$, $X$ is a vector field and $c$ is a constant, 
then $(M^n, g_0, X, c )$ is called a {\it Ricci soliton structure} and $g_0$ {\it the Ricci soliton.} 
Moreover, we say that the Ricci soliton $g_0$ is a {\it gradient Ricci soliton} 
if the vector field $X$ satisfies $X = \nabla f$ where $f$ is a function, 
and the Ricci soliton $g_0$ is a {\it non-gradient Ricci soliton} 
if the vector field $X$ satisfies $X\ne \nabla f$ for any function $f$. 
If $c$ is positive, zero, or negative,  
then $g_0$ is called a shrinking, steady, or expanding Ricci soliton, respectively.
According to \cite{CK04}, we check that a Ricci soliton is a Ricci flow solution.
\begin{prop}[see \cite{CK04}]\label{selfsimilar}
{\it A pseudo-Riemannian metric $g_0$ is a Ricci soliton if and only if 
$g_0$ is the initial metric of the Ricci flow equation,}  
$$\frac{\partial }{\partial t} g(t)_{ij} = -2 \Rc [g(t)]_{ij} \ ,$$
{\it and the solution is expressed as 
$g(t) = c(t)  (\varphi _t)^* g_0$, where $c(t)$ is a scaling parameter, 
and $\varphi _t$ is a diffeomorphism.} 
\end{prop} 
In the closed Riemannian case, Perelman \cite{P02} proved that any Ricci soliton is a gradient Ricci soliton, 
and any steady or expanding Ricci soliton is an Einstein metric with the Einstein constant zero or negative, respectively.
However in the non-compact Riemannian case, a Ricci soliton is not necessarily gradient and a steady or expanding Ricci soliton is not necessarily Einstein. 
In fact, 
any left-invariant Riemannian metric on the three-dimensional Heisenberg group 
is an expanding non-gradient Ricci soliton which is not an Einstein metric. (See \cite{BD07}, \cite{GIK06}, \cite{L07}.)
In the Riemannian case, all homogeneous non-trivial Ricci solitons are expanding Ricci solitons. 
In the pseudo-Riemannian case, there are shrinking homogeneous Ricci solitons discovered in \cite{O10}, 
while all three-dimensional homogeneous Lorentzian Ricci solitons are classified in \cite{BCGG09}. 
Other results about Lorentzian Ricci solitons are found in \cite{BBGG10}, \cite{BGG11}, \cite{CD11}. 

Lauret proved the relation of solvsolitons and Ricci solitons on Riemannian manifolds. 
\begin{theorem}[\cite{L01}]
Let $g$ be a left-invariant Riemannian metric. 
If $g$ is an algebraic Ricci soliton, then $g$ is a Ricci soliton. 
\end{theorem} 
We extend the above proposition to the pseudo-Riemannian case. 
We obtain the following  
\begin{thm}\label{thm} 
Let $g$ be a left-invariant pseudo-Riemannian metric. 
If $g$ is an algebraic Ricci soliton, then $g$ is a Ricci soliton. 
\end{thm}

\begin{proof}
Let $g$ be an algebraic Ricci soliton with $\rc = c I + D$. 
Let $\{ e_i\} _{i=1}^n$ be a pseudo-orthonormal basis. 
We define $\varphi _t$ satisfying $d\varphi _t\big|_e = e^{\frac{t}{2} D}$, and we define the vector field $X_D$, given by  
$X_D = \frac{d}{d t} \big|_0 \varphi _t (p)$. 
Then the Lie derivative of $X_D$ is given by 
$$(L_{X_D} g) (e_i, e_j)  = \frac{d}{d t} \big|_0 \varphi ^*_t g (e_i, e_j) = \dfrac{1}{2} \Big( g (D e_i , e_j) + g(e_i, D e_j)\Big), $$
 for any $i,j$. 
Therefore we obtain 
\begin{equation*}
\begin{split}
\Rc (e_i, e_j) & = \dfrac{1}{2} \Big( g \big( \rc (e_i) , e_j\big) + g\big( e_i, \rc (e_j) \big) \Big) \\ 
& = \dfrac{1}{2} \Big( g \big( (c I +D) e_i , e_j\big) + g\big( e_i, (c I +D) e_j \big) \Big) \\
& = cg  (e_i, e_j) + (L_{X_D} g)  (e_i, e_j).
\end{split} 
\end{equation*}

\end{proof}

It is natural to consider the opposite of Theorem \ref{thm}. 
A Riemannian manifold $(M, g)$ is called a solvmanifold if there
exists a transitive solvable group of isometries.
Jablonski proved 
\begin{theorem}[\cite{J11}] 
Consider a solvmanifold $(M, g)$ which is a Ricci soliton. 
Then $(M, g)$ is isometric to a solvsoliton and the transitive solvable group may be chosen to be completely solvable. 
\end{theorem} 
In the Lorentzian case, there exists a Ricci soliton which is not an algebraic Ricci soliton.
Indeed, $SL(2, \Real)$ has left-invariant Lorentzian Ricci solitons (see \cite{BCGG09}). 
But, $SL(2,\Real)$ does not have an algebraic Ricci soliton which is not Einstein (see \cite{BO11a}). 

Lauret characterized an algebraic Ricci soliton by some properties in the Riemannian setting. 
The following was proved by Lauret. 
\begin{thm}[\cite{L01}]
Let $N$ be a nilpotent Lie group with Lie algebra ${\mathfrak n}$. 
If $g$ and $g'$ are two Ricci nilsoliton metrics on $N$, then $(N, g)$ is isometric to $(N', g)$ up to scaling. 
In particular, $g' = b \varphi . g$ for some $b>0$ and some $\varphi \in \textrm{Aut} ({\mathfrak n}) $. 
\end{thm}

However, the above theorem in the Lorentzian case does not hold. 
Indeed, we prove that the metrics $g_1$ and $g_2$ are nilsolitons on $H_3$. 
They are non-isometric by \cite{R92}. 

Lauret gave a characterization of Riemannian Ricci nilsolitons on homogeneous nilmanifolds via solvable extensions. 
\begin{df}
A metric solvable extension of $( {\mathfrak n} , g)$ is a metric solvable Lie algebra 
of the form $({\mathfrak s}  = {\mathfrak a} \oplus {\mathfrak n} , \tilde{g})$ such that 
$$
[{\mathfrak s}, {\mathfrak s}]' = {\mathfrak n} = {\mathfrak a} ^{\bot}, \quad 
[X, Y]' = [X, Y], \quad 
\tilde{g} (X,Y)  = g (X,Y), \hspace{0.5cm} {\rm for \, any} \, \, X, Y\in {\mathfrak n}
$$
where $[\cdot, \cdot ]'$ and $[\cdot, \cdot ]$ denotes the Lie brackets of ${\mathfrak s}$ 
and ${\mathfrak n}$ respectively. 
\end{df}

\begin{theorem}[\cite{L01}]
A homogeneous nilmanifold is a Riemannian nilsoliton if and only if ${\mathfrak n}$ admits a metric solvable extension with ${\mathfrak a}$ Abelian whose corresponding solvmanifold is Einstein.  
\end{theorem}

It is not clear to the author whether the above result holds in the Lorentzian case. 
In the Riemannian case, if $({\mathfrak n}, g)$ satisfies $\rc (g) = c I_g + D_g$, 
then Lauret constructed a metric solvable extension $({\mathfrak s}, \tilde{g})$ of $({\mathfrak n}, g)$ that is Einstein $\Rc (\tilde{g}) = c_{g} I$. 
In this case, the constant $c_{g}$ satisfies $c_{g} = -{\rm tr} D^2_g /{\rm tr} D_g$, and ${\rm tr} D_g >0$.  
However, in the Lorentzian case there exists a solvsoliton that satisfies $c_{g} = 0, {\rm tr} D_g =0$ (see Section 4).

\section{The classical Heisenberg groups}
Let $n$ be a positive integer, and set $N = 2 n +1$. 
The higher-dimensional classical Heisenberg group $H_N$ is defined as the group of $(n+2)\times (n+2)$ upper triangular matrices
\begin{eqnarray*}
\left(
\begin{array}{ccc}
1 & {}^{t} \vec{a} & c \\
\vec{0} & I_n & \vec{b} \\
0 & ^{t} \vec{0} & 1 \\
\end{array}
\right) \ , 
\end{eqnarray*} 
where $\vec{a}, \, \vec{b} \in \Real^n, \, c \in \Real$.   
Topologically, $H_N$ is diffeomorphic to $\Real ^N$ under the map
\begin{eqnarray*}
H_3 \ni \left(
\begin{array}{ccccc}
1 & x_1 & \cdots & x_n & x_N \\
0 & 1 & \cdots & 0 & x_{1+n} \\
\vdots & \vdots & \ddots & \vdots & \vdots \\
0 & 0 & \cdots & 1 & x_{n+n} \\
0 & 0 & \cdots & 0 & 1 \\
\end{array}
\right) \mapsto (x_1, \cdots, x_n, x_{1+n}, \cdots, x_{2 n}, x_N) \in \Real^N .
\end{eqnarray*} 
Under this identification, left multiplication by $(\vec{a}, \vec{b}, c)$ corresponds to the map 
$$L_{(\vec{a}, \vec{b}, c)} (\vec{x}, \vec{y}, z)  = (\vec{a} + \vec{x} , \vec{b} + \vec{y} , c + z+ ^t \vec{a} \vec{y} ) \ .$$
Then the Lie algebra of $H_N$ has a basis consisting of
\begin{equation*}
F_i = \frac{\partial }{\partial x_i} , \quad  
F_{i+n} = \frac{\partial }{\partial x_{i+n}} + x_i \frac{\partial }{\partial x_N}, \quad  
F_N = \frac{\partial }{\partial x_N} ,
\end{equation*} 
for which 
$$[F_i, F_j] = [F_i, F_N] = [F_{i+n}, F_{j+n}] = [F_{i+n}, F_N] = 0, \, [F_i, F_{j+n}] = \delta _{i j} F_N. $$
We consider the left-invariant Lorentzian metric $g_2$ 
given by  
\begin{align*}
g_2 (F_1, F_1) = \cdots = g_2 (F_{2 n}, F_{2 n}) = -g_2 (F_N, F_N)= 1. 
\end{align*}
We remark that the above metric generalizes $g_2$ on $H_3$ in Section 5. 
In this section, we prove the following theorem. 

\begin{theorem}
The above metric $g_2$ is a nilsoliton. 
There exists a metric solvable extension which is an Einstein metric.
\end{theorem}

\proof
The Levi-Civita connection of the metric $g_2$ is given by
\begin{eqnarray}
(\nabla _{F_\alpha } F_{\beta}) = \left(
\begin{array}{ccc}
O & (\frac{1}{2} \delta _{i j}  F_N) & (\frac{1}{2} F_{i+n}) \\
(-\frac{1}{2} \delta _{i j}  F_N) & O & (-\frac{1}{2} F_i) \\
(\frac{1}{2} F_{i+n}) & (-\frac{1}{2} F_i ) & O 
\end{array}
\right) 
\end{eqnarray} 
and the non-zero components of the Ricci tensor are  
$$ \Rc (F_i, F_i) = \Rc (F_{i+n},  F_{i+n}) =\frac{3}{4} n -\frac{1}{4}, \quad \Rc (F_N, F_N)  = \frac{n}{2},$$ 
where $i = 1, \cdots, n$.  
Its Ricci operator $\rc$ is given by  
\begin{eqnarray}
\rc =\left(
\begin{array}{cc}
\Big( \dfrac{3}{4} n -\dfrac{1}{4} \Big) I_{2 n} & O \\
O & -\dfrac{n}{2}  
\end{array}
\right) . 
\end{eqnarray} 

Next we compute the Lie algebra of derivations of ${\mathfrak g}$. 
By definition, we get 
\begin{eqnarray}
D =\left(
\begin{array}{ccccc}
D^1_1 & D^1_2 & \cdots & D^1_{2 n} & D^1_N \\
D^2_1 & D^2_2 & \cdots & D^2_{2 n} & D^2_N \\
\vdots & \vdots & \ddots & \vdots & \vdots \\
D^n_1 & D^n_2 & \cdots & D^{n}_{2 n} & D^n_N \\
\vdots & \vdots & \ddots & \vdots & \vdots \\
D^{2 n}_1 & D^{2 n}_2 & \cdots & D^{2 n}_{2 n} & D^{2 n}_N \\
D^N_1 & D^N_2 & \cdots & D^N_{2 n} & D^N_N \\
\end{array}
\right) , 
\end{eqnarray}
with 
\begin{subequations}
  \begin{align}
D_N^j & = D_N^{j+n} = 0, \\
D_N^N & = D^i_i + D_{i+n}^{i+n}, \\
D_i^{j+n} & = D_j^{i+n}, \\
D_i^j + D_{i+n}^{j+n} & = 0, \\
D_{i+n}^j & = D_{j+n}^i, \\
D_N^i & = D_N^{i+n} = 0 .
  \end{align}
\end{subequations} 
If $g$ satisfies 
the nilsoliton equation (\ref{Soliton}), then it satisfies 
\begin{eqnarray} 
\begin{cases}
c & = 2 n-\dfrac{1}{2} \\
D^i_i & = -\dfrac{5}{4} n + \dfrac{1}{4} \\
a^i_j & = 0 \quad otherwise. 
\end{cases} 
\end{eqnarray} 
Therefore $g_2$ is a nilsoliton.

Next we consider a solvable extension of $H_N$. 
Its algebra has generators $H$, $F_1$, $F_2$, $F_3$ and Lie brackets  
\begin{subequations}
  \begin{align*}
[H, F_i]' & = D F_i, \quad 
[F_i, F_j]' = [F_i, F_N]' = [F_{i+n}, F_{j+n}]' = [F_{i+n}, F_N]' = 0, \\ 
[F_i, F_{j+n}]' & =  \delta _{i j} F_N, 
  \end{align*}
\end{subequations} 
where $D$ satisfies the nilsoliton equation (\ref{Soliton}). 
We consider the left-invariant pseudo-Riemannian metric $\tilde{g}$ given by 
\begin{align*}
\tilde{g} (H,H)= a, \quad 
\tilde{g}_2 (F_1, F_1) = \cdots = \tilde{g}_2 (F_{2 n}, F_{2 n}) = -\tilde{g}_2 (F_N, F_N)= 1. 
\end{align*}
Then the Levi-Civita connection of the metric $\tilde{g}$ is given by 
\begin{eqnarray}
\left(
\begin{array}{cccc}
O & O & O & O \\
(-b F_i) & \Big(\dfrac{b}{a} \delta _{i j} H\Big) & \Big(\dfrac{1}{2} \delta _{i j} F_N\Big) & \Big(\dfrac{1}{2} F_{j+n}\Big) \\
(-b F_{i+n}) & \Big(-\dfrac{1}{2} \delta _{i j} F_N\Big) & \Big(\dfrac{b}{a} \delta _{i j} H\Big) & \Big(-\dfrac{1}{2} F_{j}\Big) \\
(-2 b F_N) & \Big(\dfrac{1}{2} F_{i+n} \Big) &  \Big(-\dfrac{1}{2} F_{j}\Big) & -2 \dfrac{b}{a} H
\end{array}
\right) , 
\end{eqnarray} 
where $b= -\dfrac{5}{4} n + \dfrac{1}{4}.$ 
And the non-zero components of the Ricci tensor are 
$$\Rc (H, H) = (-2 n-4) b^2, \quad \Rc (F_i, F_i) = \Rc (F_{i+n}, F_{i+n}) = (-2 n-2) \dfrac{b^2}{a} + \dfrac{1}{2},$$  
 $$\quad \Rc (F_N, F_N) =  (-2 n-2) \dfrac{b^2}{a} + \dfrac{1}{2}. $$

Therefore it is Einstein if and only if $a=-4 b^2$ and its Einstein constant is $\dfrac{n}{2} +1$. 
\qedhere

\begin{rem}
According to Theorem \ref{thm}, the left-invariant Lorentzian metric $g_2$ is a Ricci soliton. 
We will exhibited the Ricci soliton equation by derivation $D$. 
According to the proof of Theorem \ref{thm}, we define $\varphi _t $ by
$$
e^{\frac{t}{2} D} =
\begin{pmatrix}
e^{(-\frac{5}{8} n + \frac{1}{8}) t} I_{2 n}  & O \\
 O & e^{(-\frac{5}{4} n + \frac{1}{4}) t}
\end{pmatrix}
= d\varphi _t |_e. 
$$
Then we get 
\begin{eqnarray*}
\frac{\partial \varphi ^i_t}{\partial x_i} & = & e^{(-\frac{5}{8} n + \frac{1}{8}) t}, 
\frac{\partial \varphi ^{i+n}_t}{\partial x_{i+n}} = e^{(-\frac{5}{8} n + \frac{1}{8}) t}, 
\frac{\partial \varphi ^N_t}{\partial x_N} = e^{(-\frac{5}{4} n + \frac{1}{4}) t}, \\
\end{eqnarray*}
where $i = 1, \cdots , n$. 
We can solve above differential equation, and we obtain  
$$\varphi _t (p) = (e^{(-\frac{5}{8} n + \frac{1}{8}) t} x_1, \cdots ,  e^{(-\frac{5}{8} n + \frac{1}{8}) t} x_{2 n} \ , \quad e^{(-\frac{5}{4} n + \frac{1}{4}) t} x_N) $$
So the Ricci soliton vector field is exhibited as 
\begin{eqnarray*}
X_D & = & \frac{d}{d t} \Big| _{t=0} \varphi _t (p) \\
& = & \left( -\frac{5}{8} n + \frac{1}{8}\right) \left( x_i F_i +x_{i+n} F_{i+n}\right) +\left( -\frac{5}{4} n + \frac{1}{4}\right) \left( x_N -\frac{1}{2} \sum _i x_i x_{i+n}\right) F_N.
\end{eqnarray*}
It is easy check that $\nabla _i X_j -\nabla _j X_i \ne 0, $
therefore the left-invariant Lorentzian metric $g_1$ is a non-gradient Ricci soliton. 
\end{rem}

\section{The Oscillator groups}
In this section, we will see an example of a steady solvsoliton.  
The oscillator algebra ${\mathfrak g} _m (\lambda ) = {\mathfrak g} (\lambda _1, \cdots, \lambda _m )$ has $(2 m+2)$-generator $P$, $X_1$, $\cdots$, $X_m$, $Y_1$,  $\cdots$, $Y_m$, $Q$, and Lie brackets 
$$[X_i, Y_j] = \delta_{i j} P, \quad [Q, X_j] = \lambda _j Y_j, \quad [Q, Y_j] = -\lambda _j X_j. $$
That is, ${\mathfrak g} _m (\lambda )$ is the semidirect product of the Heisenberg algebra ${\mathfrak h} _m$ generated
by $(P, X_1, \cdots, X_m, Y_1, \cdots, Y_m)$, and the line generated by $Q$, under the homomorphism
${\rm ad} | {\mathfrak h} _m : \langle Q\rangle \rightarrow {\rm Der} ({\mathfrak h} _m)$. 
It is a solvable non-nilpotent Lie algebra and the connected simply connected Lie group whose Lie algebra is ${\mathfrak g} _m (\lambda )$ is {\it the oscillator group} $G_m(\lambda) = G(\lambda _1, \cdots, \lambda _m)$.

We consider the left-invariant metric defined by 
$g_{\varepsilon } (P, P) = g_{ \varepsilon }(Q, Q) = \varepsilon $, $g_{ \varepsilon }(X_i, X_j) = g_{ \varepsilon } (Y_i, Y_j) = \delta _{i j}$, 
and other components are zero. 
If $\varepsilon = 0$ and $\lambda _i = 1$ for each $i = 1, \cdots, m$, the corresponding Lorentzian metric is also right-invariant. 
In other cases, $g_{\varepsilon}$ is not bi-invariant (see \cite{GO99}). 
The Levi-Civita connection of the metric $g_{ \varepsilon }$ is given by
\begin{equation}
\begin{split}
\nabla _P X_j & =  -\dfrac{ \varepsilon }{2} Y_j= \nabla _{X_j} P, \, \, \nabla _{X_j} Q = -\dfrac{1}{2} Y_j, \\
\nabla _Q X_j & =  \Big( \lambda _j -\dfrac{1}{2} \Big) Y_j, \, \, \nabla _P Y_j = \dfrac{\varepsilon }{2} X_j = \nabla _{Y_j} P, \\
\nabla _{Y_j} Q & =  \dfrac{1}{2} X_j, \, \, \nabla _Q Y_j = - \Big(\lambda _j -\dfrac{1}{2} \Big) X_j, \\
\nabla _{X_j} Y_j & =  \dfrac{1}{2} P = - \nabla _{Y_j} X_j, 
\end{split} 
\end{equation}
and other components are zero. 
Its Ricci tensor is expressed by
\begin{eqnarray*}
\Rc (P, P) & = & \dfrac{\varepsilon ^2 m}{2}, \quad 
\Rc (P, Q) = \dfrac{\varepsilon m}{2}, \quad 
\Rc (X_j, X_j) = -\dfrac{\varepsilon }{2}, \\ 
\Rc (Y_j, Y_j) & = & -\dfrac{\varepsilon }{2}, \quad 
\Rc (Q, Q) = \dfrac{m}{2},
\end{eqnarray*}
and other components are $0$.
Its Ricci operator $\rc$ is given by 
\begin{eqnarray}
\rc =\left(
\begin{array}{cccc}
\dfrac{\varepsilon  m}{2} & O & O & \dfrac{m}{2} \\
O & -\dfrac{\varepsilon }{2} I_m & O &O \\
O & O & -\dfrac{\varepsilon }{2} I_m &O \\
O & O & O & O
\end{array}
\right) .
\end{eqnarray}

Next we compute the Lie algebra of derivations of ${\mathfrak g}$. 
We get 
$D \in \textrm{Der} \big( {\mathfrak g} _m (\lambda )\big)$ given by  
\begin{eqnarray*} 
D P & = & \alpha P, \quad 
D X_j = a_j P + \sum_{l} b^l_j X_l +\sum_{l} c^l_j Y_l,  \\
D Y_j & = & \tilde{a}_j P + \sum_{l} B^l_j X_l + \sum_{l} C^l_j Y_l, \quad 
D Q = \mu P -\sum_{l} \lambda _l a_l X_l -\sum_{l} \lambda _l \tilde{a}_l Y_l,  
\end{eqnarray*}
with 
\begin{eqnarray}
\alpha = b^i_i +C^i_i, \quad 
b^k_j + C_k^j = 0, \quad 
\lambda _j B^i_j = -\lambda _i c^i_j ,\quad \\
\lambda _j C^i_j = \lambda _i b^i_j ,\quad 
\lambda _j b^i_j = \lambda _i C^i_j ,\quad 
\lambda _j c^i_j = -\lambda _i B^i_j, 
\end{eqnarray}
for any $i, j\ne k$. 
We prove 
\begin{thm}
If $g_{\varepsilon}$ satisfies 
the solvsoliton equation (\ref{Soliton}), then it satisfies 
\begin{eqnarray} 
\begin{cases}
\varepsilon = c  = 0, \\
\mu  = \dfrac{m}{2}, \\
\alpha = a_j = \tilde{a}_j = b^i_j = c^i_j = 0, 
\end{cases} 
\end{eqnarray} 
for any $i, j$.
Therefore the left-invariant Lorentzian metric $g_0$ is a steady solvsoliton. 
\end{thm} 

The solvable extension of the oscillator algebra ${\mathfrak g} _m (\lambda ) = {\mathfrak g} (\lambda _1, \cdots, \lambda _m )$ has $(2 m+3)$-generator $H$, $P$, $X_1$, $\cdots$, $X_m$, $Y_1$,  $\cdots$, $Y_m$, $Q$, and Lie brackets 
$$[H, Q] = D Q = \dfrac{m}{2} P, \quad [X_i, Y_j] = \delta_{i j} P, \quad  [Q, X_j] = \lambda _j Y_j, \quad [Q, Y_j] = -\lambda _j X_j, $$
where $D$ satisfies the solvsoliton equation (\ref{Soliton}). 
We consider the metric $\tilde{g}$, given by $\tilde{g} (H, H) = a$ and $\tilde{g} (F_i, F_j) = \delta _{i j}$. 
Then the Levi-Civita connection of $\tilde{g}$ is given by 
\begin{equation}
\begin{split}
\nabla _Q H & = -\dfrac{m}{2} P, \quad 
\nabla _{X_i} Y_j = \dfrac{\delta_{i j}}{2} P = - \nabla _{Y_j} X_j, \quad 
\nabla _{X_i} Q  = -\dfrac{1}{2} Y_i, \\
\nabla _Q X_j & =  \Big( \lambda _j -\dfrac{1}{2} \Big) Y_j, \quad 
\nabla _{Y_i} Q  = \dfrac{1}{2} X_i, \\
\nabla _Q Y_j & =  -\Big( \lambda _j -\dfrac{1}{2} \Big) X_j, \quad 
\nabla _Q Q = \dfrac{m}{2 a} H, 
\end{split} 
\end{equation}
and other components are zero. 
The curvature tensor is given by
\begin{eqnarray}
R(Q, X_i) X_j & = & \dfrac{\delta _{i j}}{4} P, \quad 
R(Q, Y_i) Y_j = \dfrac{\delta _{i j}}{4} P, \\
R(X_i, Q) Q & = & \dfrac{1}{4} X_i, \quad 
R(Y_i, Q) Q = \dfrac{1}{4} Y_i. 
\end{eqnarray}
Its Ricci tensor vanishes except $\Rc (Q, Q)  = \dfrac{m}{2}$.
Hence this is not Einstein. 

\begin{rem} 
Next, we consider solvable extensions of $G_m (\lambda )$. 
However, there are derivations of the Lie algebra such that the solvable extension of $G_m (\lambda )$ is Einstein. 
Actually, when $m=1$, we consider the derivation of the Lie algebra given by 
\begin{align*} 
[H, P] & = D P = 2 b P, \, [H, X] = a P + b X + c Y, \, [H, Y] = \tilde{a} P - c X + b Y,  \\ 
[H, Q] & = k P - \lambda a X - \lambda \tilde{a} Y, \, [X, Y] = P, \, [Q, X] = \lambda Y, \, [Q, Y] = -\lambda X. 
\end{align*} 

Then the Levi-Civita connection of the metric $g_1$ is given by 
\begin{align} 
\nabla _H P & = -\nabla _P H = b P, \quad  \nabla _H X = \dfrac{\lambda + 1}{2} a P + c Y, \\
\nabla _X H & = \dfrac{\lambda - 1}{2} a P -b X, \quad 
 \nabla _H Y  = \dfrac{\lambda + 1}{2} \tilde{a} P - c Y, \\  
\nabla _X H & = \dfrac{\lambda - 1}{2} \tilde{a} P -b X, \quad 
 \nabla _H Q = -\dfrac{\lambda + 1}{2} a X -\dfrac{\lambda + 1}{2} \tilde{a} Y - b Q, \\
\nabla _Q H & = -k P + \dfrac{\lambda - 1}{2} a X +\dfrac{\lambda - 1}{2} \tilde{a} Y -b Q, \quad 
\nabla _P Q = \nabla _Q P = \dfrac{b}{h} H, \\
\nabla _ X X & = \nabla _Y Y = \dfrac{b}{h} H, \quad 
\nabla  _X Y = -\nabla _Y X = \dfrac{1}{2} P, \\
\nabla _X Q & = \dfrac{1- \lambda}{2 h} a H -\dfrac{1}{2} Y, \quad
\nabla _Q X = \dfrac{1- \lambda}{2 h} a H +\Big( \lambda -\dfrac{1}{2} \Big) Y, \\
\nabla _Y Q & = \dfrac{1- \lambda}{2 h} \tilde{a} H +\dfrac{1}{2} X, \quad 
\nabla _Q Y = \dfrac{1- \lambda}{2 h} \tilde{a} H -\Big( \lambda -\dfrac{1}{2} \Big) X, \\
\nabla _Q Q & = \dfrac{k}{h} H. 
\end{align} 

The non-zero components of the Ricci tensor are 
\begin{align} 
\Rc (H, H) & = -4 b^2, \quad  \Rc (X, X) = \Rc (Y, Y) = \Rc (P, Q) = -\dfrac{4 b^2}{h}, \quad \\
\Rc (Q, Q) & = \dfrac{1-\lambda ^2}{2 h} (a^2 +\tilde{a} ^2) - \dfrac{2 k b}{h} +\dfrac{1}{2},  \\
\Rc (X, Q) & = \dfrac{\lambda -1}{2 h} (3 a b +\tilde{a} c),  \quad 
\Rc (Y, Q) = \dfrac{\lambda -1}{2 h} (3 \tilde{a} b -a c).  
\end{align} 
When $\lambda = 1$, the metric $\tilde{g}$ is Einstein if and only if $h=4 k b$ and $k b\ne 0$. 
When $\lambda \ne 1$, the metric $\tilde{g}$ is Einstein if and only if $h=4 k b$, $a = \tilde{a} =0$ and $k b\ne 0$, 
or $b = c=0$, $h=-(\lambda^2 -1) (a^2 + \tilde{a}^2)$, $a \tilde{a} \ne 0$. 
\end{rem}

\section{Three-dimensional left-invariant Lorentzian Ricci solitons} 
In this section, we consider left-invariant Lorentzian metrics on $3$-dimensional Lie groups. 
The Ricci solitons that we consider in this section were constructed in $\cite{O10}$. 

\subsection{Nilsolitons on $H_3$}
Let $H_3$ be the $3$-dimensional Heisenberg group. 
N. Rahmani and S. Rahmani \cite{RR06} proved that any left-invariant Lorentzian metric on $H_3$ is classified into three types 
$g_1$, $g_2$ and $g_3$, up to isometry and scaling, given by  
\begin{align*}
g_1 & =  -dx^2 +dy^2 +(x\ dy +dz)^2 \ , \\
g_2 & =  dx^2 +dy^2 -(x\ dy +dz)^2 \ , \\
g_3 & =  dx^2 +(x \ dy +dz)^2 -\big( (1-x) dy -dz\big) ^2 \ .
\end{align*} 
Nomizu \cite{N79} proved that $g_3$ is flat. 
N. Rahmani and S. Rahmani \cite{RR06} showed that $g_1$ and $g_2$ are not Einstein. 
In \cite{O10}, we proved that the left-invariant Lorentzian metric $g_1$ is a Lorentzian Ricci soliton. 
My previous paper \cite{O10} contains a mistake that $g_2$ is Einstein. 
I correct it and introduce the following theorem. 

\begin{theorem}[\cite{N79, O10}]
On the $3$-dimensional Heisenberg group, the metrics $g_1$ and $g_2$ are shrinking non-gradient Ricci solitons,  and $g_3$ is flat. 
\end{theorem}


In this section, we prove the following theorem. 
\begin{theorem}
The above metrics $g_1$ and $g_2$ are nilsolitons.
There exist metric solvable extensions that are Einstein metrics.
\end{theorem}

\proof 
The metric $g_2$ was done in Section 3.  
We consider the metric $g_1$. 
Let our frame be defined by
$$F_1 = \frac{\partial }{\partial z} \ , 
F_2 = \frac{\partial }{\partial y} -x\frac{\partial }{\partial z} \ , 
F_3 = \frac{\partial }{\partial x} \ ,$$
and coframe 
$$\theta ^1  = x\ dy + dz \ , 
\theta ^2  = dy \ , 
\theta ^3  = dx  \ .$$
It is easy to check that the metric $g_1$ is represented as 
$$g_1 =(\theta ^1)^2+(\theta ^2)^2-(\theta ^3)^2 \ , $$ 
and all brackets $[F_i , F_j ]$ vanish except $[F_2, F_3] = F_1$ . 

Then the Levi-Civita connection of the metric $g_1$  is given by  
\begin{eqnarray}
(\nabla _{F_i} F_j)=\frac{1}{2} \left(
\begin{array}{cccc}
0 & F_3 & F_2 \\
F_3 & 0 & F_1 \\
F_2 & -F_1 & 0 \\
\end{array}
\right) , 
\end{eqnarray} 
and its Ricci tensor is expressed as 
$$ \Rc (F_1, F_1)  = -\frac{1}{2} , \quad \Rc (F_2, F_2)  = \frac{1}{2}, \quad \Rc (F_3, F_3)  = -\frac{1}{2}, \quad $$
and other components are $0$.
Obviously the Lorentzian metric $g_1$ is not Einstein. 
Its Ricci operator $\rc$ is given by   
\begin{eqnarray}
\rc = \frac{1}{2} \left(
\begin{array}{ccc}
-1 & 0 & 0 \\
0 & 1 & 0 \\
0 & 0 & 1 
\end{array}
\right) . 
\end{eqnarray} 

Next we compute the Lie algebra of derivations of ${\mathfrak g}$. 
Then we get 
$$\textrm{Der} ({\mathfrak g}) = 
\left\{ \left(
\begin{array}{ccc}
a^2_2 + a^3_3 & a^1_2 & a^1_3 \\
0 & a^2_2 & a^2_3 \\
0 & a^3_2 & a^3_3 
\end{array}
\right) \right\} . 
$$

If $g$ satisfies 
the nilsoliton equation (\ref{Soliton}), then it satisfies 
\begin{eqnarray} 
\begin{cases}
c & = \dfrac{3}{2} \\
a^2_2 & = a^3_3 = -1 \\
a^i_j & = 0 \quad otherwise. 
\end{cases} 
\end{eqnarray} 
Therefore $g_1$ is a nilsoliton.

Next we consider a solvable extension of $(H_3, g_1)$. 
Its algebra has generators $H$, $F_1$, $F_2$, $F_3$ and Lie brackets 
\begin{eqnarray*}
[H, F_1]' & = & D F_1=-2 F_1, \quad [H, F_2]' =  D F_2=-F_2, \\ 
\,  [H, F_3]' & = & D F_3=- F_3, \quad [F_2, F_3]'  =  F_1, 
\end{eqnarray*}
where $D$ satisfies the solvsoliton equation (\ref{Soliton}). 
We consider the left-invariant pseudo-Riemannian metric $\tilde{g}$ given by 
$\tilde{g} = h (\theta ^0)^2 +(\theta ^1)^2+(\theta ^2)^2-(\theta ^3)^2.$
Then the Levi-Civita connection of the metric $\tilde{g}$ is given by
\begin{eqnarray}
(\nabla _{F_i} F_j) = \left(
\begin{array}{cccc}
0 & 0 & 0 & 0 \\
2 F_1 & -\dfrac{2}{h} H & \dfrac{1}{2} F_3 & \dfrac{1}{2} F_2 \\
F_2 & \dfrac{1}{2} F_3 & -\dfrac{1}{h} H & \dfrac{1}{2} F_1 \\
F_3 & \dfrac{1}{2} F_2 & - \dfrac{1}{2} F_1 & -\dfrac{1}{h} H 
\end{array}
\right) , 
\end{eqnarray} 
and the non-zero components of the Ricci tensor are 
$$\Rc (H, H) = -6, \quad \Rc (F_1, F_1) = -\dfrac{8}{h}-\dfrac{1}{2},  \quad \Rc (F_2, F_2) = -\Rc (F_3, F_3) = -\dfrac{4}{h}+\dfrac{1}{2}.  $$
Therefore it is Einstein if and only if $h=-4$ and its Einstein constant is $\dfrac{3}{2}$. 
\qedhere

According to Theorem \ref{thm}, the left-invariant Lorentzian metric $g_1$ is a Ricci soliton (see also \cite{O10}). 





\subsection{A solvsoliton on $E(2)$}
Let $E(2)$ be the group of rigid motions of Euclidean $2$-space.
We consider the left-invariant Lorentzian metric given by 
$$g_1 =  dx^2 + (\cos x \ dy + \sin x \ dz)^2  -( -\sin x \ dy + \cos x \ dz )^2. 
$$ 
In \cite{O10}, we proved that this metric $g_1$ is a Lorentzian Ricci soliton. 
In this section, we prove the following
\begin{theorem}
The metric $g_1$ is a solvsoliton.
There exists a metric solvable extension which is an Einstein metric.
\end{theorem}

\proof 
Let our frame be defined as
$$
F_1  = \frac{\partial }{\partial x} \ ,  \ 
F_2 = \cos x \frac{\partial }{\partial y} + \sin x \frac{\partial }{\partial z}  \ , \ 
F_3  = -\sin x \frac{\partial }{\partial y} + \cos x \frac{\partial }{\partial z} \ ,
$$
and coframe 
$$
\theta ^1  = dx \ ,  \ 
\theta ^2  = \cos x \ dy +\sin x \ dz \ , \
\theta ^3  = -\sin x \ dy +\cos x \ dz \ .
$$
It is easy to check that the metric $g_1$ is represented as  
$$g_1 =(\theta ^1)^2+(\theta ^2)^2-(\theta ^3)^2 \ , $$ 
with 
$$[F_1, F_2] = F_3 \ , \ [F_2, F_3] = 0 \ , \ [F_3, F_1] = F_2. \ $$
Then the Levi-Civita connection of the metric $g_1$ is given by
\begin{eqnarray}\label{levi3}
(\nabla _{F_i} F_j) = \left(
\begin{array}{cccc}
0 & 0 & 0 \\
-F_3 & 0 & -F_1 \\
F_2 & -F_1 & 0 \\
\end{array}
\right) , 
\end{eqnarray} 
and its Ricci tensor vanishes except $\Rc(F_1, F_1)  = 2$ .
Its Ricci operator $\rc$ is given by 
\begin{eqnarray}
\rc = \left(
\begin{array}{ccc}
2 & 0 & 0 \\
0 & 0 & 0 \\
0 & 0 & 0
\end{array}
\right) . 
\end{eqnarray} 

Next we compute the Lie algebra of derivations of ${\mathfrak g}$. 
We get 
$$\textrm{Der} ({\mathfrak g}) = 
\left\{ \left(
\begin{array}{ccc}
0 & 0 & 0 \\
a^2_1 & a^2_2 & a^2_3 \\
a^3_1 & -a^2_3 & a^3_3
\end{array}
\right) \right\} . 
$$

If $g$ satisfies 
the solvsoliton equation (\ref{Soliton}), then it satisfies 
\begin{eqnarray} 
\begin{cases}
c & = 2 \\
a^2_2 & = a^3_3 = -2 \\
a^i_j & = 0 \quad otherwise. 
\end{cases} 
\end{eqnarray} 
Therefore $g_1$ is a solvsoliton.

Next we consider a solvable extension of $E(2)$. 
Its algebra has generators $H$, $F_1$, $F_2$, $F_3$ and Lie brackets
\begin{eqnarray*}
[H, F_i]' = D F_i, \quad [F_1, F_2]' = F_3, \quad [F_2, F_3]' = 0, \quad [F_3, F_1]' = F_2,
\end{eqnarray*}
where $D$ satisfies the solvsoliton equation (\ref{Soliton}). 
We consider the left-invariant pseudo-Riemannian metric $\tilde{g}$ given by 
$\tilde{g} = h (\theta ^0)^2 -(\theta ^1)^2+(\theta ^2)^2+(\theta ^3)^2.$
Then the Levi-Civita connection of the metric $\tilde{g}$ is given by
\begin{eqnarray}
(\nabla _{F_i} F_j) = \left(
\begin{array}{cccc}
0 & 0 & 0 & 0 \\
0 & 0 & 0 & 0 \\
2 F_2 & -F_3 & -\dfrac{2}{h} H & - F_1 \\
2 F_3 & F_2 & - F_1 & \dfrac{2}{h} H 
\end{array}
\right) , 
\end{eqnarray} 
and the non-zero components of the Ricci tensor are 
$$\Rc (H, H) = -8, \quad \Rc (F_1, F_1) = 2,  \quad \Rc (F_2, F_2) = -\dfrac{8}{h},  \quad \Rc (F_3, F_3) = \dfrac{8}{h}. $$
Therefore it is Einstein if and only if $h=-4$ and its Einstein constant is $2$. 
\qedhere

According to Theorem \ref{thm}, the left-invariant Lorentzian metric $g_1$ is a Ricci soliton (see also \cite{O10}). 




\subsection{A solvsoliton on $E(1,1)$}
Let $E(1, 1)$ be the group of rigid motions of Minkowski $2$-space.
We consider the left-invariant Lorentzian metric given by 
$$g_1 = -dx^2 +(e^{-x} dy + e^{x} dz )^2 + (e^{-x} dy - e^{x} dz )^2.$$
In \cite{O10} we proved that the above metric $g_1$ is a Lorentzian Ricci soliton. 
In this section, we prove the following 
\begin{theorem}
This metric $g_1$ is a solvsoliton.
There exists a metric solvable extension which is an Einstein metric.
\end{theorem}

\proof
Let our frame be defined as
$$
F_1 = \frac{\partial }{\partial x} \ ,  \ 
F_2 = \dfrac{1}{2} \Big( e^{x} \frac{\partial }{\partial y} + e^{-x}\frac{\partial }{\partial z} \Big) \ , \ 
F_3 = \dfrac{1}{2} \Big( e^{x} \frac{\partial }{\partial y} - e^{-x} \frac{\partial }{\partial z} \Big)  ,
$$
and coframe 
$$
\theta ^1  = dx, \quad  
\theta ^2  = e^{-x} \ dy +e^x \ dz, \quad  
\theta ^3  = e^{-x} \ dy -e^x \ dz.  
$$
It is easy to check that the metric $g_1$ is represented as
$$g_1 =-(\theta ^1)^2+(\theta ^2)^2+(\theta ^3)^2 \ , $$ 
with 
$$[F_1, F_2] = F_3 \ , \ [F_2, F_3] = 0 \ , \ [F_3, F_1] = -F_2. \ $$
Then the Levi-Civita connection of the metric $g_1$ is given by 
\begin{eqnarray}\label{levi3}
(\nabla _{F_i} F_j) = \left(
\begin{array}{cccc}
0 & 0 & 0 \\
-F_3 & 0 & -F_1 \\
F_2 & -F_1 & 0 \\
\end{array}
\right) , 
\end{eqnarray} 
and its Ricci tensor vanishes except $\Rc (F_1, F_1)  = -2$.
Obviously the Lorentzian metric $g_1$ is not Einstein. 
Its Ricci operator $\rc$  is given by 
\begin{eqnarray}
\rc = \left(
\begin{array}{ccc}
2 & 0 & 0 \\
0 & 0 & 0 \\
0 & 0 & 0
\end{array}
\right) . 
\end{eqnarray} 

Next we compute the Lie algebra of derivations of ${\mathfrak g}$. 
We get 
$$\textrm{Der} ({\mathfrak g}) = 
\left\{ \left(
\begin{array}{ccc}
0 & 0 & 0 \\
a^2_1 & a^2_2 & a^2_3 \\
a^3_1 & a^2_3 & a^3_3
\end{array}
\right) \right\} . 
$$

If $g$ satisfies 
the solvsoliton equation (\ref{Soliton}), then it satisfies 
\begin{eqnarray} 
\begin{cases}
c & = 2 \\
a^2_2 & = a^3_3 = -2 \\
a^i_j & = 0 \quad otherwise. 
\end{cases} 
\end{eqnarray} 
Therefore $g_1$ is a solvsoliton.

Next we consider a solvable extension of $E(1,1)$. 
Its algebra has generators $H$, $F_1$, $F_2$, $F_3$ and Lie brackets 
\begin{eqnarray*}
[H, F_i]' = D F_i, \quad [F_1, F_2]' = F_3, \quad [F_2, F_3]' = 0, \quad [F_3, F_1]' = -F_2, 
\end{eqnarray*}
where $D$ satisfies the solvsoliton equation (\ref{Soliton}). 
We consider the left-invariant pseudo-Riemannian metric $\tilde{g}$ given by 
$\tilde{g} = h (\theta ^0)^2 -(\theta ^1)^2+(\theta ^2)^2+(\theta ^3)^2.$
Then the Levi-Civita connection of the metric $\tilde{g}$ is given by  
\begin{eqnarray}
(\nabla _{F_i} F_j) = \left(
\begin{array}{cccc}
0 & 0 & 0 & 0 \\
0 & 0 & 0 & 0 \\
2 F_2 & -F_3 & -\dfrac{2}{h} H & - F_1 \\
2 F_3 & F_2 & - F_1 & \dfrac{2}{h} H 
\end{array}
\right) , 
\end{eqnarray} 
and the non-zero components of the Ricci tensor are 
$$\Rc (H, H) = -8, \quad \Rc (F_1, F_1) = -2,  \quad \Rc (F_2, F_2) = -\dfrac{8}{h},  \quad \Rc (F_3, F_3) = -\dfrac{8}{h}.  $$
Therefore it is Einstein if and only if $h=-4$ and its Einstein constant is $2$. 
\qedhere

According to Theorem \ref{thm}, the left-invariant Lorentzian metric $g_1$ is a Ricci soliton (see also \cite{O10}). 


\section*{Acknowledgments} 
The author gratefully acknowledges many helpful comments from Professor Ryoichi Kobayashi.
The author would like to thank Professor Ohnita, Professor Hiroshi Tamaru and Professor Parker, for support and comments. 
The authors wish to thank the referees' valuable comments and feedback on our manuscript.

\end{document}